\documentclass[letterpaper, 10 pt, conference]{IEEEtran}
\usepackage{lipsum}
\IEEEoverridecommandlockouts
\usepackage{cite}
\usepackage{amsmath,amssymb,amsfonts}
\usepackage{graphicx}
\usepackage{textcomp}
\usepackage{xcolor}
\usepackage{verbatim}

\usepackage[ruled]{algorithm}
\usepackage{algorithmic}

\makeatletter
\newcommand\fs@norules{\def\@fs@cfont{\bfseries}\let\@fs@capt\floatc@ruled
  \def\@fs@pre{}%
  \def\@fs@post{}%
  \def\@fs@mid{\kern3pt}%
  \let\@fs@iftopcapt\iftrue}
\makeatother
\floatstyle{norules}
\restylefloat{algorithm}

\usepackage{hyperref}

\usepackage{amsthm}
\usepackage{amssymb}
\usepackage{amsmath}

\usepackage{xcolor}
\usepackage{tikz}
\usetikzlibrary{calc,shapes.geometric,arrows,positioning,intersections}
\usepackage{pagecolor}
\usetikzlibrary{decorations, decorations.text,backgrounds}

\usepackage{algorithm}
\usepackage{algorithmic}

\newtheorem{theorem}{Theorem}

\DeclareMathOperator{\kl}{KL}

\floatstyle{norules}
\restylefloat{algorithm}

\newcommand{\bV}{\mathbb{V}}

\newcommand{\cN}{{\cal N}}

\DeclareMathOperator*{\argmin}{arg\,min}

\def\BibTeX{{\rm B\kern-.05em{\sc i\kern-.025em b}\kern-.08em
    T\kern-.1667em\lower.7ex\hbox{E}\kern-.125emX}}
\begin{document}
\title{Power Grid Reliability Estimation via Adaptive Importance Sampling}
\author{Aleksander Lukashevich, \IEEEmembership{Student  Member, IEEE}, and Yury Maximov, \IEEEmembership{Member, IEEE}
\thanks{The work of Yury Maximov at LANL is supported by the U.S. DOE Advanced Grid Modeling program as a part of ``Robust Real-Time Control, Monitoring, and Protection of Large-Scale Power Grids in Response to Extreme Events''  and LANL LDRD~projects. }
\thanks{Aleksander Lukashevich is with 
the Center for Energy Science and Technology, Skolkovo Institute of Science and Technology, 
Moscow, Russia (e-mail: Aleksandr.Lukashevich@skoltech.ru)}
\thanks{Yury Maximov is with the Theoretical Division, Los Alamos National Laboratory, 
Los Alamos, NM 87545 USA (e-mail: yury@lanl.gov)}}

\maketitle

\begin{abstract}
    Electricity production currently generates approximately 25\% of greenhouse gas emissions in the USA. Thus, increasing the  amount of renewable energy is a key step to carbon neutrality. However, integrating a large amount of fluctuating renewable generation is a significant challenge for power grid operating and planning. Grid reliability, i.e. an ability to meet operational constraints under power fluctuations, is probably the most important of them. In this paper, we propose computationally efficient and accurate methods to estimate the probability of failure, i.e. reliability constraints violation, under a known distribution of renewable energy generation. To this end, we investigate an importance sampling approach, a flexible extension of Monte-Carlo methods, which adaptively changes the sampling distribution to generate more samples near the reliability boundary. The approach allows to estimate failure probability in real-time based only on a few dozens of random samples, compared to thousands required by the plain Monte-Carlo. Our study focuses on high voltage direct current power transmission grids with linear reliability constraints on power injections and line currents. We propose a novel theoretically justified physics-informed adaptive importance sampling algorithm and compare its performance to state-of-the-art methods on multiple  IEEE power grid test cases.
\end{abstract}


\section{Introduction}
\label{sec:introduction}

Carbon-free electricity generation is one of the most vital global challenges for the next decades. Because of their ecological and economic benefits, renewable energy sources, such as wind, hydro, and solar power generation, become more demanded, accessible, and widely used in modern power grids  \cite{gielen2019role,harjanne2019abandoning}. For instance, California's renewable portfolio standard currently requires 33\% of retail electricity sales to come from renewable resources, and will require 60\% by 2030, and 100\% by~2045~\cite{golden2003senate}.
However, renewable energy generation is highly volatile and brings significant uncertainty to power systems. This also gives rise to many challenges for power system operators trying to integrate renewables into power grids \cite{schmietendorf2017impact,liang2016emerging}. In particular, power systems operational policies and reliability assessment must be verified over various additional uncertainties, including increased variation in power generation and disturbances. 

Various algorithms have been developed so far for ensuring grid reliability. Some of them are based on machine learning, and utilize historical data about weather, renewables' generation, and grid operating parameters to estimate the risk of failure and influence of uncertainty due to whether changes~\cite{zhang2017data,contingencyMLE}.
Requiring large datasets and high data collecting time to make an accurate prediction, machine learning methods become impractical for real-time operation if a large disturbance, contingency, or a sudden operational policy change precedes the reliability assessment. Another class of algorithms is based on analytical approximation of the failure probability~\cite{nemirovski2007convex}. In this approach, the risk of interest is upper-bounded by an integral of an appropriate function that admits analytical or numerical computation. However, even in the simplest case of linear reliability constraints and Gaussian fluctuations of renewables, existing approaches tend to overestimate the risk. Moreover, for a sufficiently rare event, the risk overestimation for these algorithms can be infinitely high~\cite{nemirovski2007convex} which compromises their practical~efficiency. 

Finally, algorithms based on sampling values of power generated by renewables and approximating the grid failure probability by its empirical counterpart often provide a valuable alternative for accessing the reliability posture of a power system. Monte-Carlo (MC), hybrid and Markov Chain MC have been earlier applied to risk-based reliability assessment of transmission power grids~\cite{su2005probabilistic,vittal2009steady,yu2009probabilistic,chen2008probabilistic}. 
In~\cite{montecarlofeasibility}, the authors exploited Monte Carlo simulation for estimating failure probability and interpreted the risk by classifying it into low, medium and high risk operating points. Variations of the load-flow solution due to renewables fluctuation, nodal and line parameters uncertainty were considered in~\cite{su2005probabilistic,yu2009probabilistic}.  An inverse problem of wind turbine controls to meet reliability margins with high-probability is discussed in \cite{vittal2009steady}.  A comprehensive survey of sampling-based methods for power systems reliability assessment is given in~\cite{chen2008probabilistic}.
Unfortunately, these algorithms explore the space of fluctuations uniformly, which dramatically 
reduces their performance in understanding and evaluating the effect of a rare event such as a severe disturbance. 

Importance sampling is a valuable alternative to Monte-Carlo sampling, which allows adjusting distribution for generating more samples in the areas of interest, e.g., close to the reliability boundary. Pmvnorm~\cite{genz2020package} is one of the most efficient importance sampling algorithms in general, but its performance is often limited for rare events probability estimation~\cite{owen2019importance}, which is of the utmost importance to power systems study.  ALOE~\cite{owen2019importance}~is another efficient method designed especially for computing a rare event probability. However, it does not fully respect the geometry of reliability constraints. 
It thus requires a large number of samples to estimate the risk of failure, especially for large power grids and multi-line failures. Finally, a convex optimization-based algorithm for adaptive importance sampling from exponential families was proposed in~\cite{ryu2014adaptive}. At each step, the algorithm adjusts the distribution parameters so that the sampler's variance is minimized. Unfortunately, the distribution of output power that leads to a failure is far from the exponential family,  limiting the algorithm's efficiency in power systems.

This paper proposes an adaptive importance sampling method to efficiently estimate the risk of reliability constraints violation.
We present an importance sampling algorithm that uses physical information to generate a mixture of distributions to sample from and then uses convex optimization to iteratively adjust the weights of the mixture. Our algorithm substantially improves static weights assignment of ALOE~\cite{owen2019importance} when reliability constraints are highly correlated.
The approach allows to address the risk estimation problem in real-time even for large power grids with a small failure probability. We theoretically analyze the accuracy and complexity of our algorithm for the case of Gaussian power fluctuations from renewables; however, the technique is not limited to the Gaussian case. Finally, we evaluate the performance of our sampling methods over multiple real and synthetic test cases and compare it to the state-of-the-art. 

The paper is organized as follows. 
In Section \ref{sec:setup} we present the failure probability estimation problem and introduce notation used in the paper. We outline the importance sampling algorithm and present its theoretical analysis in Section \ref{sec:prob}. Empirical study and comparison to the state-of-the-art are given in Section~\ref{sec:emp}. In Section~\ref{sec:conclusion} we conclude  with a brief summary and discussion on possible applications of our results. 

\section{Background and Problem Setup}\label{sec:setup}

Being a popular load flow model, the higher-voltage direct current (DC) model remains simple for the analysis because of linear relations between power injections and phase angles. Let $G = (V, E)$ be a power grid graph with a set of buses $V$, $|V| = n$ and a set of lines $E$, $|E| = m$. Let $p \in \mathbb{R}^{n}$ and $\theta\in\mathbb{R}^{n}$ be vectors of power injections and phase angles respectively.  The power system is balanced, e.g., the sum of all power injections equals zero $\sum_{i \in V} p_i = 0$. To avoid ambiguity let $s$ be the slack bus and $\theta_s = 0.$ Let $B \in \mathbb{R}^{n \times n}$ be an admittance matrix of the system, $p = B\theta$. The components $B_{ij}$ are such that $B_{ij} \neq 0$ if there is a line between buses $i$ and $j$ and  for any node $B_{ii} = - \sum_{j\neq i} B_{ij}$, e.g., $B$ is a Laplacian matrix.  Let $B^{\dagger}$ be the pseudo-inverse of $B$, $\theta = B^{\dagger} p$. The DC power flow equations, generation and reliability constraints are then
\begin{subequations}
\label{eq:DC-PF}
    \begin{equation}
        p = B \theta
        \label{eq:DC-PF-a}
    \end{equation}
    \begin{equation}
        \underline{p}_i \leq p_i \leq \overline{p}_i
        \label{eq:DC-PF-b}, i \in V \text{ and } |\theta_i - \theta_j| \leq \bar{\theta}_{ij}, \; (i,j)\in E
    \end{equation}
\end{subequations}
Reliability constraints~\eqref{eq:DC-PF-b} 
define a polytope $P$ in the space of power injections, $P = \{p: W p \le b\}$, so that the reliability constraints are violated if and only if power injections $p \not\in P.$ 

To derive an explicit expression of matrix $W$ we consider the incidence matrix $A$, such that for any buses $i$ and $j$ with $i<j$ connected by an edge $k$, $A_{ki} = 1$, $A_{kj} = -1$ and all other elements in row $k$ are equal to zero. Then the phase angle constraints are $AB^{\dagger}p \le \bar\theta$, $-AB^{\dagger}p \le \bar\theta$. 
Finally, as the slack bus balances the system, $p_s = -\sum_{i\neq s} p_i$ let $C \in \mathbb{R}^{n\times n}$ be a symmetric matrix such that for any non-slack buses $i$ and $j$, and the slack bus $s$, $C_{ii} = 1$, $C_{ij} = 0$, $C_{ss}=0$, and $C_{is} = -1$. In other words, $C p$ is a vector of grid power injections expressed in terms of non-slack injections only, since the slack bus power injection is fully determined by the other ones.

Finally, from Eqs.~\eqref{eq:DC-PF} the following system of inequalities defines the reliability polytope, $P = \{p: Wp \le b\}$,
\begin{equation}
(AB^\dagger C, - A B^\dagger C, C, -C)^\top p \le (\bar\theta, \bar\theta, \overline{p}, \underline{p})^\top, 
\label{eq:feasibility_ineqs}
\end{equation}
where $W = (AB^\dagger C, - A B^\dagger C, C, -C)^\top$ and $b = (\bar\theta, \bar\theta, \overline{p}, \underline{p})^\top$. Let $J = 2m + 2n$ be a number of constraints, e.g., rows in matrix $W$, then the reliability polytope $P$ is $\bigl\{p:\; \bigcap_{i=1}^J p^\top\omega^i \le b_i\bigr\}$. 

Stochastic uncertainty in renewable generation and power consumption imposes a question of power grid reliability, e.g., estimating a probability that at least one of the reliability constraints is violated. Namely, we consider Gaussian fluctuations of power injections $p$ with known mean $\mu$ and covariance $\Sigma$ and aim at computing a failure probability $\Pi$:
\begin{align}\label{eq:prob}
    \Pi = \mathbb{P}(p\not\in P) = \int_{\mathbb{R}^n} \upsilon(p) 1[p\not\in P] d p, \; p\sim \cN(\mu, \Sigma), 
\end{align}
where $\mathbb{P}$ is a probability taken w.r.t. the normal pdf $\upsilon(p)$ of $p$. 

Notice, that the probability $\Pi$ does not have an analytic expression, is computationally intractable, and even hard to approximate~\cite{owen2019importance,ryu2014adaptive,cappe2008adaptive, khachiyan1989problem}. 
In practice, a union bound is often used to upper bound $\Pi$. Let $\Pi_i$ be a probability of a single event, e.g., $p^\top\omega^i > b_i$. It has an explicit expression for the Gaussian distribution, and by union bound inequality
$\sum_{i\le J}\Pi_i \ge \Pi \ge \max_{i\le J} \Pi_i$; however, the bounds are loose when dealing with correlated failures which is often the case for power systems. 

To refine the failure estimate and take into account simultaneous violation of multiple constraints, we propose an importance sampling procedure that allows to count the average number of constraints $N$ violated at the same time and, thus, improve the failure probability estimation to $\Pi/N$ instead of $\Pi$. It is meaningful for large power grids where multiple events are likely to happen synchronously.

Table~\ref{tab:notation} summarizes paper's notation. We use lower indices for elements of vectors and matrices, lower-case letters for probability density functions (pdfs), and upper-case letters for cumulative distribution functions (cdfs). When it does not lead to confusion, we use $\mathbb{P}$, $\mathbb{E}$, and $\mathbb{V}$ to denote  probability, expectation, and variance without explicitly mentioning a distribution. 

\begin{table}[t]
    \centering
    \caption{Paper notation. }
    \begin{tabular}{l|l|l|l}
        $E$ & set of lines, $|E| = m$ & $\upsilon(p)$ & nominal distribution pdf\\
        $V$ & set of buses, $|V| = n$ & $\upsilon_D$ & parametric distribution pdf\\
        $B$ & $n\times n$ admittance matrix& $x$ & mixture distribution para- \\
        $p_i$ & power injection & & meters, $x\in X \subseteq \mathbb{R}^J$\\
        $\underline{p}_i$ & lower generation limit & 
        $D_i$& $p\!\sim\!\cN(\!\mu, \!\Sigma)$ s.t. $p^\top\! \omega^i \!> \!b_i$\hspace{-5mm}\\
        $\overline{p}_i$ & upper generation limit & $N$ & number of samples\\
        $\theta_i$ & phase angle & $\mathbb{P}$, $\mathbb{E}$ & probability, expectation\\
        $\theta_{ij}$ & phase angle difference
        &$\mathbb{V}$, $\kl$ & Variance, KL-divergence\\
        $\bar\theta_{ij}$ & angle difference limits & ${\cal N}\!(\mu, \!\Sigma)\hspace{-5mm}$ &  Gaussian distribution with\\
        $I_n$ & $n\times n$ identity matrix & & mean $\mu$ and covariance $\Sigma$\\
        $J$ & number of constraints & $\Phi$ & ${\cal N}(0,1)$ distribution cdf\\ 
        ${P}$ & reliability set, $p\!:\! W p \!\le\! b\!\!\!\!$ & $U(0,1)$ & uniform $(0,1)$ distribution\\
        $\omega^i$ & rows of matrix $W$, $i\le J$ & $\Pi$ & failure probability, $p\not\in P$
    \end{tabular}
    \label{tab:notation}
    \vspace{-4mm}
\end{table}

\section{Failure Probability Estimation}\label{sec:prob}

\subsection{A Single Constraint Case}
We start with estimating the probability of fluctuating power injections to cause a failure of an individual constraint, e.g., $\Pi_i = \mathbb{P}(p^\top \omega^i \ge b_i)$ for some $i$, $1\le i \le J$. In the case of a Gaussian distribution, $p\sim \cN(\mu,\Sigma)$ there is a closed form expression for it:
\begin{align*}
\Pi_i & = \mathbb{P}_{p\sim \cN(\mu, \Sigma)}(p^\top \omega^i \ge b_i)  = \mathbb{P}((p - \mu)^\top \omega^i \ge b_i - \mu^\top \omega^i)\\
& = \mathbb{P}_{{\tilde p}\sim \cN(0, I_n)} \left((\Sigma^{1/2}\omega^i)^\top {\tilde p} \ge b_i - \mu^\top \omega^i\right) \\
& = \mathbb{P}_{{\tilde p}\sim \cN(0, I_n)} \left( {\tilde p}^\top {\bar\omega}_i \ge (b_i - \mu^\top \omega^i)/\|\Sigma^{1/2}\omega^i\|_2\right)  \\
& = \Phi((b_i - \mu^\top \omega^i)/\|\Sigma^{1/2}\omega^i\|_2), 
\end{align*}
where $\Phi$ is the standard Gaussian distribution cdf, ${\tilde p} = \Sigma^{-1/2}(p-\mu)$, ${\bar\omega}_i = (\Sigma^{1/2}\omega^i)^\top/\|\Sigma^{1/2}\omega^i\|_2$, and $1\le i \le J$. $\Pi_i$ are the probabilities of $p^\top \omega^i \ge b_i$, so that 
$
    \Pi = \mathbb{P}(\exists i: p^\top \omega^i \ge b_i) \le \sum_{i\le J}\mathbb{P}(p^\top \omega^i \ge b_i) = \sum_{i\le J} \Pi_i.
$

Algorithm~\ref{alg:sample1d} is an instance of the inverse transorm method \cite{l2009monte} which allow to sample $p \sim \cN(\mu, \Sigma)$ s.t. $p^\top \omega^i > b_i$. We refer this distribution as $D_i$, and its pdf is $\upsilon(p)/\Pi_i$ if $p^\top \omega^i > b_i$ and $0$ otherwise. Notice, that sample $p \sim \cN(\mu, \Sigma)$ can be obtained with the plain MC from $\cN(\mu, \Sigma)$, but it requires on average $1/\Pi_i$ trials instead of just one for Algorithm~\ref{alg:sample1d}.

\begin{algorithm}[H]
 \caption{Sampling $p\sim\cN(\mu,\Sigma)$ conditioned on $p^\top \omega^i \geq b_i$}
 \label{alg:sample1d}
 \begin{algorithmic}[1]
 \renewcommand{\algorithmicrequire}{\textbf{Input:}
 }
 \REQUIRE Mean $\mu$, covariance $\Sigma$, and a constraint $p^\top \omega^i \le b_i$.
 \renewcommand{\algorithmicensure}{\textbf{Output:}}
 \ENSURE  $p\sim\cN(\mu,\Sigma)$ s.t. $p^\top \omega^i \ge b_i$
  \STATE  Sample $z \sim \cN(0, I_n)$ and $u \sim U(0,1)$
  \STATE  Compute $y = \Phi^{-1}(\Phi(\tau) + u(1 - \Phi(\tau)))$
  \STATE Set $\phi = \bar\phi y + (I_n - \bar\phi\phi^\top) z$, with $\bar\phi = \Sigma^{1/2} \omega^i / \|\Sigma^{1/2} \omega^i\|_2$
  \RETURN $p = \Sigma^{1/2} (\phi+\mu)$
 \end{algorithmic}
 \end{algorithm}
 
 \subsection{Multiple Constraints Case}
 
The case of multiple constraints is more involved. Indeed, there is no analytical formula for a failure probability and, moreover, its exact computation is intractable~\cite{khachiyan1989problem}. 
Monte-Carlo sampling, $p\sim \cN(\mu, \Sigma)$ is inefficient in estimating the failure probability, especially if it is small. Indeed, it requires on average $O(1/\Pi)$ samples to get at least one of the outside the reliability polytope, $p\not\in P$. 

The importance sampling idea is to change the distribution one samples from and assign a weight to each sample to account for the~change: 
\begin{align*}
    \Pi & = \mathbb{P}(p\not\in P) = \int_{\mathbb{R}^n} f(p) \upsilon(p) dp \\
    & = \!\int_{\mathbb{R}^n}\!\! \frac{f(p)\upsilon(p)}{\upsilon_D(p,x)}\upsilon_D(p,x) dp \approx \frac{1}{N}\sum_{i=1}^N \frac{f(p^i)\upsilon(p^i)}{\upsilon_D(p^i,x)}, p^i \sim \upsilon_D(p,x),
\end{align*}
where we refer to $\upsilon(p)$ as \emph{nominal} distribution, and $\upsilon_D(p,x)$ as \emph{synthetic} distribution with parameter $x$, and $f(p) = 1[p\not\in P]$. 

A natural extension of importance sampling with a single linear constraint to the case of multiple linear constraints is to sample from a mixture distribution: 
\begin{align}\label{eq:mix}
    D = \sum_{i \le J} x_i D_i, \text{ with } \sum_{i\le J} x_i = 1, x_i \ge 0, \; 1\le i \le J,
\end{align}
where $D_i$ is $\cN(\mu,\Sigma)$ conditioned on $p^\top\omega^i > \!b_i$. The sampling algorithm consists of two steps. First, we choose a distribution $D_i$ with probability $x_i$. Second, we sample $p \sim D_i$, i.e. $p\sim \cN(\mu,\Sigma)$ given $p^\top\omega^i>b_i$, according to Algorithm~\ref{alg:sample1d}.

Probability density function $\upsilon_D(p, x)$ of $D$ is given by
\begin{align*}
    \upsilon_D(p, x) = \begin{cases}
    0, & \!\!p \in P,\\
    \sum_{i\le J} x_i \upsilon_i(p) 1[p^\top \omega^i > b_i], & \!\!p\not\in P,
    \end{cases}
\end{align*}
where $1[\cdot]$ is an indicator of an event, and 
\begin{align*}
    \upsilon_i(p) = \frac{\upsilon(p)}{\Phi((b_i - \mu^\top\omega^i)/\|\Sigma^{1/2}\omega^i\|_2)}  1[p^\top\omega^i > b_i],
\end{align*}
In contrast to the classical Monte-Carlo, which explores the uncertainty space uniformly according to the nominal distribution, importance sampling from parametric distribution $\upsilon_D(p,x)$ yields samples only from the area of interest, i.e., $p\not\in P$. 
More specifically, Monte-Carlo generates many samples from the true distribution of power injections to estimate failure probability, while the proposed approach only samples power injections that lead to a failure and adjusts their weights.
Figure \ref{fig:00} illustrates the difference. 

\begin{figure}
    \centering
    \vspace{-3mm}
    \includegraphics[width=.5\textwidth]{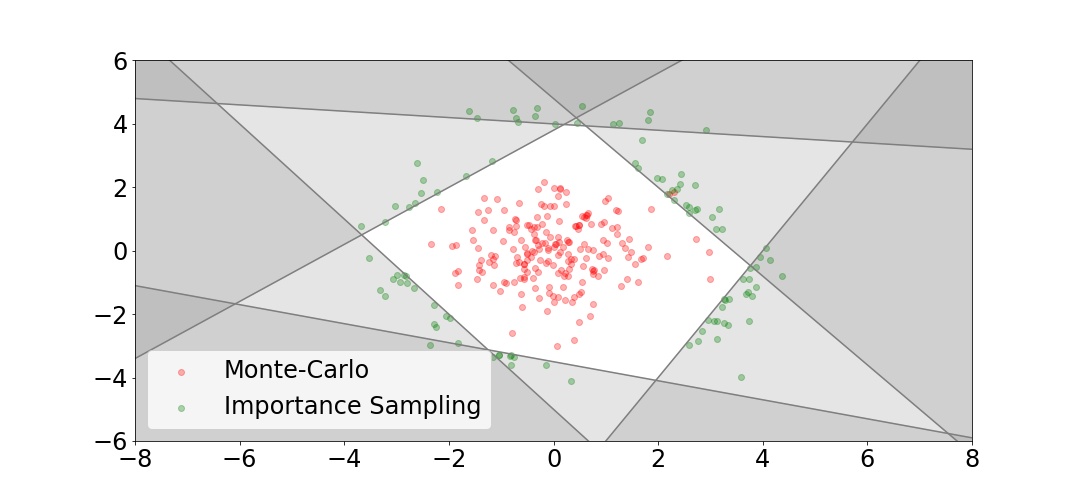}
    \vspace{-6mm}
    \caption{White area stands for generations which do not lead to a failure. A set of generations leading to at least one constraint violation is in grey. Two or more reliability constraints are not satisfied in dark grey area. Samples from a nominal distribution and the constructed mixture marked in red and green resp. 
    }
    \label{fig:00}
    \vspace{-5mm}
\end{figure}

Having a distribution mixture $\{D_i\}_{i=1}^J$we are looking for the weights $\{x_i\}_{i=1}^J$ to approximate a distribution $D$, $p \sim \cN(\mu, \Sigma)$ s.t. $p\not\in P$, in the optimal way. Note that, any positive $\{x_i\}_{i=1}^J$ weights lead to an unbiased estimate 
\begin{align}\label{eq:aloe}
    {\hat \Pi} = \frac{1}{N}\sum_{i=1}^N \frac{\upsilon(p^i)}{\upsilon_D(p^i, x^i)}, \quad p^i \sim D^{x^i}
\end{align}
of the probability $\Pi$. Here $x^i$ is the vector of mixture weights for $i^{\textup{th}}$ sample. Indeed, by linearity of the  expectation $\mathbb{E}_{\upsilon_D(p,x)} \hat \Pi = \frac{1}{n}\sum_{i\le n} \int \frac{\upsilon(p^i)1[p^i\not\in P]}{\upsilon_D(p^i, x)}\upsilon_D(p^i, x) dp = \Pi$. 

Despite being unbiased for any $x$ with positive components, variance of the  estimate~\eqref{eq:aloe} highly depends on the choice of $x$.
In \cite{owen2019importance} the authors suggested to take $x_i \propto \Pi_i$. 
While it leads to a consistent estimate, the estimator's variance is still high, especially when violation of multiple constraints is likely to happen in the system \cite{owen2019importance}. In practice, it leads to high sample complexity of the estimator which compromises its real time application. In the next subsection, we significantly improve the sampler's efficiency by using convex optimization to find the optimal combination of the mixture distribution weights $x$. 

\subsection{Convexity of Importance Sampling Variance}

We will measure the effectiveness of our estimator by its mean squared error, which is equal to the variance since the estimator is unbiased.
The importance sampler variance is 
\begin{align*}
    \bV_{\upsilon_D(p, x)}\left(\frac{\upsilon(p) f(p)}{\upsilon_D(p, x)}\right) & = \int_{p\in\mathbb{R}^n} \left(\frac{f(p)\upsilon(p)}{\upsilon_D(p, x)} - \Pi\right)^2 \upsilon_D(p, x) d p \\
    & = \int_{p\in\mathbb{R}^n} \frac{f^2(p)\upsilon^2(p)}{\upsilon_D(p, x)} d p - \Pi^2 =: V(x), 
\end{align*}
where $f(p) = 1[p\not\in P]$. 

The optimal synthetic distribution can be chosen to minimize the variance, $\upsilon_*(p) = f(p)\upsilon(p)/\Pi$, and thus provide a better approximation to the integral. Notice, that for $\upsilon_D = \upsilon_*$ the variance $V(x) = 0$ and attains its minimum.  However, if $\upsilon_*(p)$ does not belong to the parametric family $\{\upsilon_D(p,x)\}_x$, we are looking for the best approximation of $\upsilon_*(p)$ within it, i.e. a minimum possible value of $V(x)$. 

Figure~\ref{fig:fscheme} illustrates our approach. Starting from an initial weight assignment, $x_i^1 \propto \Pi_i$, at each iteration $N$ of the algorithm we sample $p^k \sim \upsilon_D(\cdot, x^k)$ and compute $x^{k+1}$ to minimize the variance of the estimate. We also update an empirical estimate to the probability 
\begin{align}
    \hat \Pi = \frac{1}{k}\sum_{t=1}^k \frac{\upsilon(p^t) f(p^t)}{\upsilon_D(p^t, x^t)}, \; p^t \sim \upsilon_D(\cdot, x^t), t\le k \label{eq:emp}
\end{align}
and update parameters $x$ based on the value of $V(x^N)$ and its gradient. Before discussing the update strategy for the parameters $x$, we  outline some important properties of the variance $V(x)$ and empirical estimate~\eqref{eq:emp}. 

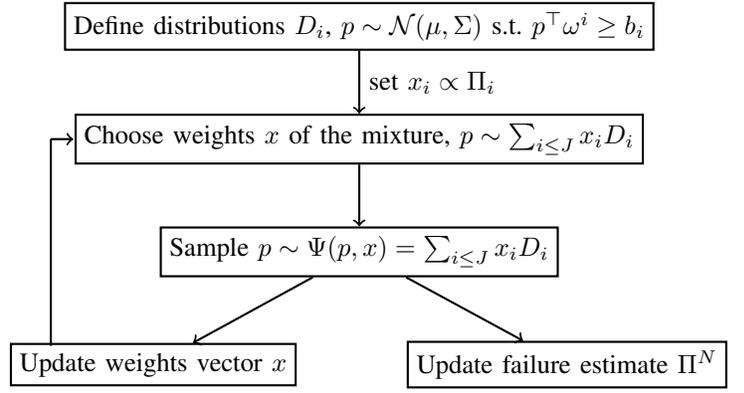
\begin{figure}
    \centering
    \begin{tikzpicture}[node distance={15mm}, thick, main/.style = {draw}] 
    \node[main] (a)  {Define distributions $D_i$, \par $p\sim \cN(\mu, \Sigma)$ s.t. $p^\top \omega^i\ge b_i$}; 
    \node[main] (b) [below of=a] {Choose weights $x$ of the mixture, $p\sim\sum_{i\le J}x_i D_i$}; 
    \node[main] (c) [below of=b] {Sample $p \sim \Psi(p, x) = \sum_{i\le J }x_i D_i$}; 
    \node (x) [below of=c]{};
    \node[main] (d) [left =  7 mm of x] {Update weights vector $x$};
    \node[main] (e) [right = 5 mm of x] {Update failure estimate $\Pi^N$}; 
    \draw[->] (a) -- (b) node[midway, right] {set $x_i\propto \Pi_i$};
    \draw[->] (b) -- (c);
    \draw[<-] (-3.8,-1.5) -- (-4.1,-1.5);
    \draw[-] (-4.1,-4.25) -- (-4.1,-1.5);
    \draw[->] (c) -- (d); 
    \draw[->] (c) -- (e);
    \end{tikzpicture} 
    \caption{Scheme of the proposed algorithm. First, we initialize distributions $D_i$ based on a given set of constraints. Second, we initialize the weights of a mixture distribution to sample  $x_i\propto \Pi_i$. Then, with each new sample from $\sum_{i\le J} x_i D_i$ we update the weight vector $x$ and a failure estimate $\Pi_n$. See Eq.~\eqref{eq:md-upd} for details.}
    \label{fig:fscheme}
    \vspace{-4mm}
\end{figure}

Theorem~\ref{thm:var-convexity} implies convexity of the variance minimization problem 
\begin{align}
    \min_x \; & V(x), \; 
    \text{s.t. } \sum_{i=1}^J x_i = 1, x_i \ge 0, 1\le i \le J\label{eq:var-min}
\end{align}
in $x$ for the mixture distribution $\upsilon_D(p,x) = \sum_{i=1}^J x_i \upsilon_i(p)$. To improve numerical stability one may also add constraints $x_i \ge \varepsilon > 0$ that guarantee that the variance is bounded.

\begin{theorem}\label{thm:var-convexity}
Optimization problem~\eqref{eq:var-min} is convex in $x$. Moreover, 
\begin{align*}
    \nabla V(x) = \mathbb{E}_{p \sim \upsilon_D(\cdot,x)} \biggl[ - \frac{f^2 (p)\upsilon^2(p)}{\upsilon_D^2(p,x)} (\upsilon_1(p), \dots, \upsilon_J(p))^\top\biggr],
\end{align*} 
for any $x > 0$. 
\end{theorem}
\begin{proof}
The sub-integral expression, $f^2(p)\upsilon^2(p)/\upsilon_D(p,x)$, is convex for any $x$, which implies the integral's convexity. Indeed, the Hessian of the sub-integral expression is non-negative for any $x$ 
\begin{align*}
    \nabla^2_x \left[f^2(p)\upsilon^2(p)/\upsilon_D(p,x)\right] = 2\frac{f^2(p)\upsilon^2(p)}{\upsilon^3_D(p,x)} h^\top h \succeq 0,
\end{align*}
where $h = (\upsilon_1(p), \dots, \upsilon_J(p)) = \nabla_x \upsilon_D(p,x)$. Finally, by the dominated convergence theorem as $\mathbb{E}_{p \sim \upsilon_D(\cdot,x)} \nabla_x\left[f^2(p)\upsilon^2(p)/\upsilon_D(p,x)\right]$ is finite for every $x>0$, one can exchange the order of differentiation and integration and 
\begin{align*}
\mathbb{E}_{p \sim \upsilon_D(\cdot,x)} & 
\left[- \frac{f^2(p)\upsilon^2(p)}{\upsilon_D^2(p,x)}h^\top\right] = 
\mathbb{E}_{p \sim \upsilon_D(\cdot,x)}  \nabla_x\left[\frac{f^2(p)\upsilon^2(p)}{\upsilon_D(p,x)}\right] \\
& = 
    \nabla_x \mathbb{E}_{p \sim \upsilon_D(\cdot,x)}  \left[\frac{f^2(p)\upsilon^2(p)}{\upsilon_D(p,x)}\right] = \nabla V(x),
\end{align*}
which concludes the proof of the theorem. 
\end{proof}

\begin{theorem}\label{thm:unbias}
$\hat \Pi$ is an unbiased estimate of $\Pi$ if for all $k$, $1\le k \le N$, $x_k > 0$ and $x_k$ is independent of $x^j$ and $p^j$ for~$N\ge j > k\ge 1$.
\end{theorem}
\begin{proof} 
Let $f(p) = 1[p\not\in P]$. By the law of total expectation
\begin{align*}
    \mathbb{E}\, {\hat\Pi} & = \frac{1}{n}\sum_{k=1}^n  \mathbb{E}\biggl[\mathbb{E}_{p^k\sim\upsilon_D(p, x^k)} \biggl[ \frac{\upsilon(p^k) f(p^k)}{\upsilon_D(p^k,x^k)}\bigg| x^k\biggr] \biggr] = \sum_{i=1}^n \frac{\Pi}{n}= \Pi,
\end{align*}
as $\mathbb{E}_{p^k} \biggl[ \frac{\upsilon(p^k)f(p^k)}{\upsilon_D(p^k,x^k)}\bigg| x^k\biggr] = \int_{\mathbb{R}^n} \frac{\upsilon(p^k)1[p^k \not\in P]}{\upsilon_D(p^k,x^k)}\upsilon_D(p^k,x^k) dp^k = \Pi$.
\end{proof}

According to Theorem~\ref{thm:unbias}, the importance sampling estimate is unbiased. Theorem \ref{thm:var} bounds the variance of $\hat\Pi$. 

\begin{theorem}\label{thm:var}
Variance of $\hat\Pi$ equals $N^{-2}\sum_{k=1}^N V(x^k)$ if for all $k$ and $j,$ $1\le k < j \le N$, $x_k > 0$ and $x_k$ is independent of $x^j$ and~$p^j$.
\end{theorem}
\begin{proof} Let $f(p) = 1[p\not\in P]$. As $\mathbb{V} (\hat\Pi)  = \mathbb{E}(\hat\Pi - \Pi)^2$ one has 
\begin{align*}
    \mathbb{V} & = \frac{1}{N^2} \sum_{k=1}^N \mathbb{E}\biggl( \frac{\upsilon(p^k)f(p^k)}{\upsilon_D(p^k, x^i)} - \Pi\biggr)^2
    \\ & 
    \; + \frac{2}{N^2}\sum_{k < j} \mathbb{E} \biggl(\frac{\upsilon(p^k)f(p^k)}{\upsilon_D(p^k, x^k)} - \Pi\biggr)\biggl(\frac{\upsilon(p^j)f(p^j)}{\upsilon_D(p^j, x^j)} - \Pi\biggr)\\
    & = \frac{1}{N^2}\sum_{k=1}^N \mathbb{E}\biggl[\mathbb{E}\biggl[\biggl(\frac{\upsilon(p^k)f(p^k)}{\upsilon_D(p^k,x^k)} - \Pi\biggr)\big| x^k\biggr]\biggr] \\
        & + \frac{2}{N^2}\sum_{k < j} \mathbb{E}\mathbb{E}\biggl[ \biggl(\frac{\upsilon(p^k)f(p^i)}{\upsilon_D(p^i, x^k)} - \Pi\biggr)\!\!\biggl(\frac{\upsilon(p^j)f(p^j)}{\upsilon_D(p^j, x^j)} - \Pi\biggr)\big|x^k\biggr],
\end{align*}
where the latter is equal to $N^{-2}\sum_{k=1}^N V(x^k)$.
\end{proof}

In the next section, we present a numerical method that guarantees convergence of $V(x^N)$ to the optimal value $V^*$ with an additive error~$O(1/\sqrt{N})$.

\subsection{Numerical Method}\label{sec:nm}

In this section we focus on efficient numerical methods for minimizing variance which, therefore, accelerate convergence of the importance sampling procedure. The mirror descent \cite{nemirovski2009robust} is known for its efficiency for simplex-constrained minimization problems. Its particular advantage compared to the stochastic gradient descent~\cite{ryu2014adaptive} and other optimization algorithms is only a logarithmic dependence on the problem dimension.

The mirror descent update for solving 
\begin{align}\label{eq:md-problem}
    \min_{x \ge 0, \sum_{i\le J} x_i = 1} V(x)
\end{align}
is an iterative modification of a point $x_k$ according to
\begin{align}\label{eq:md-upd}
    x_{k+1} = \argmin_{\substack{x \ge 0\\ \sum_{i\le J} x_i = 1}}\left\{\eta^k \nabla V(x^k)^\top (x - x^k) + D_\omega(x, x^k)\right\},
\end{align}
where $\eta^k > 0$ is a step-size, and $D_\omega(x, x^k)$ is the Bregman divergence which is defined for any strongly convex and smooth (distance generating) function $\omega$ as
\[
    D_\omega(x, x^k) = \omega(x) - \{\omega(x^k) + \nabla \omega(x^k)^\top (x^k - x)\}.
\]
So as the distance generating function is strongly convex and smooth in $x,$ so is the Bregman divergence. When $\omega(x) = \|x\|_2^2/2,$ mirror descent step is the same as in the gradient descent method, $x^{k+1} = x_k - \eta^k \nabla V(x^k)$. However, the negative entropy, $\omega(x) = -\sum_{i=1}^n x_i \log x_i$, is known to be the optimal choice for simplex constrained optimization. Solving~Eq.~\eqref{eq:md-upd} in $x$ leads to an update 
\[x^{k+1}_i = x^k_i \frac{\exp(-\eta^k(\nabla V(x^k))_i)}{\sum_{i=1}^J x_k \exp(-\eta^k(\nabla V(x^k))_i)}, \eta^k > 0\]
for $k\ge 1$ and $1\le i \le J$. 

Finally, upon minimizing stochastic objective $V(x),$ the expectation of the gradient is inaccessible, so one substitutes $\nabla V(x)$ with a stochastic gradient that comes from the uncertainty realization~$p$,
\[
    g(x, p) =  - \frac{f^2 (p)\upsilon^2(p)}{\upsilon_D^2(p,x)} h^\top, \; \mathbb{E}_{p\sim \upsilon_D}g(x, p) = \nabla V(x),
\]
where $h = (\upsilon_1(p), \dots, \upsilon_k(p))$. Finally 
\begin{align*}
    \upsilon_D(p,x)/\upsilon(p) = \sum_{i\le d} x_i 1[p^\top \omega^i \ge b_i],
\end{align*}
and $f(p) = 1[p\not\in P] = 1 $ for any $p$ sampled from $\upsilon_D$. 
Thus 
\begin{align}\label{eq:_upd}
    x^{k+1}_i = x^k_i \frac{\exp\left(\frac{\eta^k \upsilon(p) 1[p^\top \omega^i > b_i]}{\sum_{i\le J} x_i 1[p^\top \omega^i \ge b_i]}\right)}{\sum_{i=1}^J x_i^k \exp\left(\frac{\eta^k \upsilon(p) 1[p^\top \omega^i > b_i]}{\sum_{i\le d} x_i 1[p^\top \omega^i \ge b_i]}\right)}, 1\le i \le J
\end{align}

Theorem~\ref{thm:lan} is a restatement of \cite[Theorem 4.1.]{lan2020first} which establishes the convergence rate of the mirror descent algorithm. 

\begin{theorem}\label{thm:lan}
    For any function $V(x)$ that is $M$-Lipschitz in $\ell_1$ norm, i.e. $\|V(x) - V(y)\|_\infty \le M \|x-y\|_1 \forall x,y$, a constant step-size policy $\eta^k = \eta \le 1/M$, and 
    a sequence $\{x^k\}_{k\ge 1}$ generated by \eqref{eq:md-upd} with $\omega(x) = \sum_{i\le J} x_i\log x_i$, one has
    \begin{align*}
        \frac{1}{N}\sum_{i=1}^N (V(x^i)  - V^*) \le 
        \frac{\log J + (M^2 + \sigma^2) N \eta^2}{N\eta} ,
    \end{align*}
where $\mathbb{E}_{p\sim\upsilon_D}\|g(p,x)- \nabla V(x)\|_\infty^2 \le \sigma^2$, and $V_*$ is the optimal value of Problem~\eqref{eq:md-problem}. 
\end{theorem}

In our study, function $V(x)$ is $M$-Lipschitz for $x\ge \varepsilon$ with 
\[
M \le \|\nabla V(x)\|_\infty \le \int_{\mathbb{R}^n} \frac{f^2(p)\upsilon^2(p)}{\upsilon_D(p,x)^2} \upsilon(p) dp \le \Pi/\varepsilon, 
\]
and 
\begin{align*}
    \mathbb{E}_{p\sim \upsilon_D} & \|g(p,x) - \nabla V(x)\|_\infty^2 
    = \mathbb{E} \biggl\|\frac{f^2(p)\upsilon^3(p)}{\upsilon^2(p,x)} h  - \nabla V(x)\biggr\|_\infty^2 \le \\
    & 2\mathbb{E} \biggl\|\frac{f^2(p)\upsilon^3(p)}{\upsilon^2(p,x)} h\biggr\|_\infty^2  + 2\mathbb{E} \|\nabla V(x)\|_\infty^2 \le 4 \Pi^2/\varepsilon^2. 
\end{align*}
To this end, according to Theorem~\ref{thm:lan} the optimal choice of $\eta = \varepsilon\Pi^{-1}\sqrt{\log J/(5N)} \le \varepsilon\min_{i\le d}\Pi_i^{-1}\sqrt{\log J/(5N)},$ which yields almost dimension independent convergence rate stated in Theorem~\ref{thm:md-c}. 
\begin{theorem}\label{thm:md-c} Mirror descent with stochastic update~\ref{eq:_upd} and a step-size policy $\eta^k = \eta \varepsilon\Pi^{-1}\sqrt{(\log J)/N}$,  $\eta \sqrt{N/\log J}\le 1$ yields
\begin{align*}
    \mathbb{V}_{\upsilon_D} (\hat \Pi) = \frac{1}{N}\sum_{k=1}^N V(x^k) < \frac{V^*}{N} +  \frac{\Pi \sqrt{\log J}}{\varepsilon \eta N^{3/2}} + 
    \frac{5\eta \Pi \sqrt{\log J}}{\varepsilon N^{3/2}}
    \eta, 
\end{align*}
where $V^*$ is the optimal value of Problem~\eqref{eq:md-problem}.
\end{theorem}

Compared to the earlier results of~\cite{owen2019importance}, the rate of convergence depends as $O(\sqrt{\log J})$ on the dimension $J$, while earlier results \cite{owen2019importance} claim linear at dependence. Thus, our result provides a substantial acceleration for large-scale problems.  

\section{Empirical Study}\label{sec:emp}

\subsection{Algorithms and implementation details}
We compare performance of importance samplers over real and simulated test cases whose dimensions vary from several dozens to several thousands variables. We limit the empirical setting to considering Gaussian distributions and linear constraints only.

\paragraph{Compared Algorithms} In this study, we have compared Monte-Carlo Sampling, ALOE \cite{owen2019importance}, pmvnorm \cite{genz2020package} and mirror descent for variance minimization (MD-Var). We have also applied the algorithms to the same setting with KL-divergence \cite{l2009monte} between the generated distribution $\upsilon_D$ and the optimal distribution $\upsilon^*_D$ as a measure of estimator's quality (instead of variance $V(x)$). This similarly leads to a convex optimization problem similar to~\cite{rubinstein2013cross}. The former and the latter are the proposed methods.  

%

\paragraph{Implementation details} We have used Python 3.8.5. and PandaPower~2.2.2~\cite{pandapower.2018} on MacBook Pro (2.4GHz, 8-Core Intel i9, 64 GB RAM). In the experiments computational time for each of the cases for MD-Var method have not exceeded two minutes, which makes the solution applicable for the operational practice. Our code is publicly available on Github\footnote{\url{https://github.com/vjugor1/adaptive_importance_sampling_power_grids}}. 

\subsection{Test cases and numerical results}

We evaluate our algorithms on multiple real (power grids) and simulated test cases. We estimate the probability of system failure, i.e. the probability that at least one of the realibility constraints fails. Assuming Gaussian fluctuations of output power of renewables, the  probability equals to the Gaussian volume of the reliability polytope's complement $\mathbb{R}^n\setminus P$, as it was shown earlier. First, we conduct our experiments on the regular polytope, then we consider degenerate polytope. The latter is merely two parallel planes, one of them has a number of slightly shivered duplicates. This test assesses the stability of the algorithms and ability to handle joint geometry of the problem. Finally, we apply the proposed algorithms to various power grids. 

\paragraph{Regular polytope}
We consider a regular 2 dimensional polytope with $J$ faces ($J\ge 3$) centered at zero, $
    P = \{p \in \mathbb{R}^2: \omega_j^\top x \leq \tau, 1\le j \le J\},
$
where $\omega_j = (\sin(2\pi j/J), \cos(2 \pi j/J))$. 
We assume $p\sim\cN(0, I_2)$, where $I_2$ is $2\times 2$ identity matrix. The probability $p\not\in P$ rapidly converges to $\exp(-\tau^2/2)$ as $J\to\infty$ \cite{owen2019importance}. 
Figure~\ref{fig:01} compares performance of MC,  ALOE~\cite{owen2019importance}, mirror descent (Section~\ref{sec:nm}) minimizing variance (MD-Var) and KL-divergence (MD-KL) and pmvnorm~\cite{genz2020package} methods for $\tau = 6$ and $J = 360$. Figure~\ref{fig:01} shows the histogram of $\hat \Pi/\Pi$ for 100 runs of the algorithms on a sample size of 1000.
 The MD-Var method demonstrates a slightly better performance then ALOE, while pmvnorm tends to significantly underestimate the probability of $p\not\in P$. Monte-Carlo sampling from the nominal distribution failed to generate any event $p\not\in P$ in $10^6$ tries, and estimated a failure probability $\Pi$ as zero.

\begin{figure}[t!]
    \centering
    \vspace{-3mm}
    \includegraphics[width=.45\textwidth]{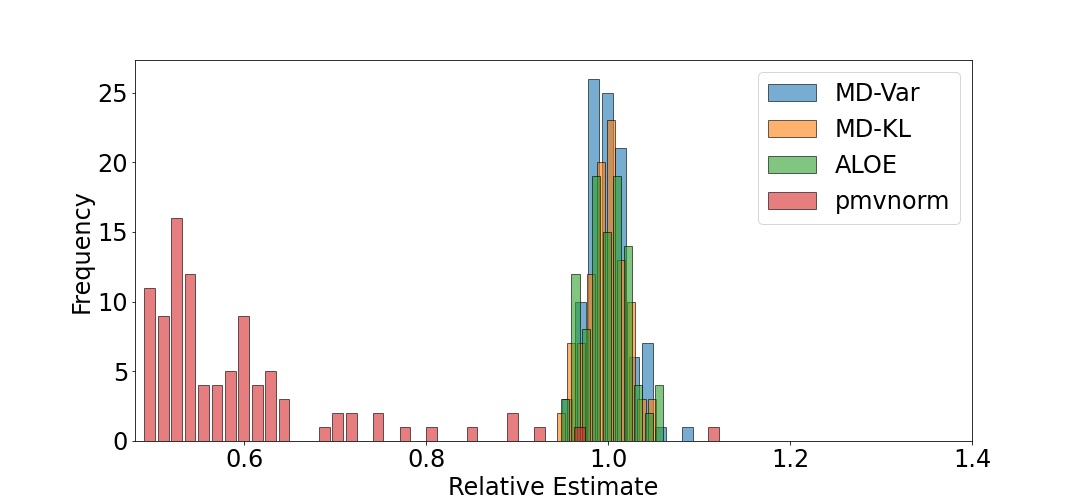}\hspace{-5mm}
    \caption{Importance 
    sampling methods performance on a 2-dimensional regular polytope with $360$ faces and failure probability $\Phi(-6)$.}
    \label{fig:01}
    \vspace{-5mm}
\end{figure}
\begin{figure}[t!]
    \centering
    \includegraphics[width=.45\textwidth]{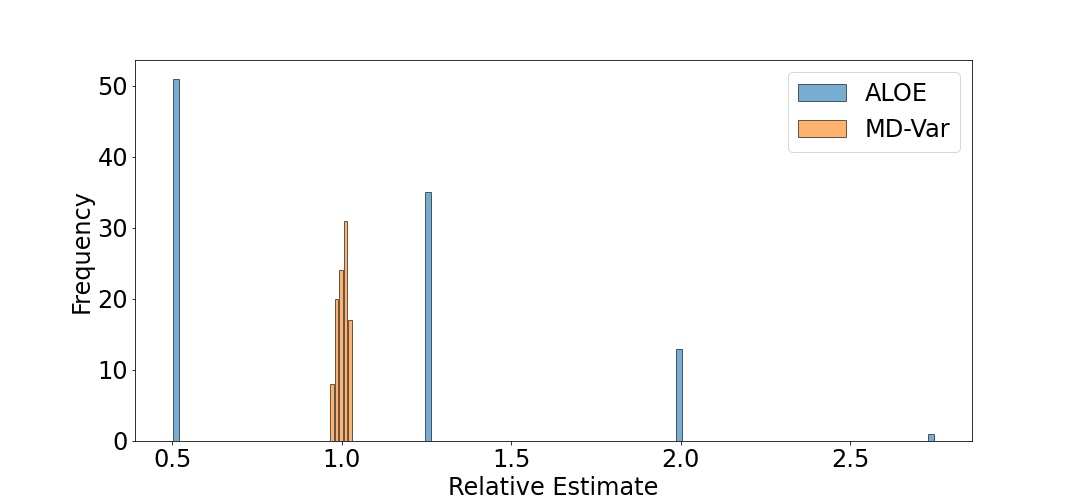}\hspace{-5mm}
    \caption{Importance 
    sampling methods performance on a 2-dimensional degenerate polytope with $1500$ faces and failure probability $2\Phi(-1)$.}
    \label{fig:degenerate_polytope}
    \vspace{-5mm}
\end{figure}


\paragraph{Degenerate polytope} Although ALOE is one of the best choices for a regular polytope, the algorithm does not take into account the joint geometry of various hyperplanes. 
As the second example we consider a degenerate polytope with $J=1500$ faces, where $\omega^1 = (0, 1)$  and $\omega^j = (\xi, -1 - \xi)$, $2\le j \le J$. We take $\xi \sim \mathcal{U}[-\varepsilon, \varepsilon]$ for small $\varepsilon = 10^{-6}$. Note that $\omega^j$ for $j\geq 2$ are almost identical. Hence probability $\Pi$ is quite close to $2\Phi(-\tau)$. In this experiment, ALOE puts a lot of efforts on sampling points in the area $\{p: (0, -1)^\top p \ge \tau\}$, while the set $\{p: \omega_1^\top p \ge \tau\}$ remains unexplored which leads to a higher variance of the sampler and a less efficient method compared to the proposed optimization approach. Figure~\ref{fig:degenerate_polytope} illustrates the performance of ALOE in this case. 

\paragraph{Power Grid Cases} In this subsection we consider real-world polytopes corresponding to DC power grids (IEEE test cases) and Gaussian power injections.
%
%
We ran the algorithms on all the test cases  accessible through PandaPower~\cite{pandapower.2018}. There were 27 cases with the number of buses varying from 4 to 9241. The proposed methods (MD-Var and MD-KL) took less than two minutes of computational time on a personal laptop for each of them. 

Table~\ref{tab:sample-compX} shows the minimal number of samples that are required by the algorithms to achieve 
\begin{align}
    \Pi/2 \le \hat\Pi \pm s(\hat\Pi) \le 3\Pi/2, \label{eq:tk}
\end{align}
where $s(\hat\Pi)$ is the empirical standard deviation of the estimate. This ensures that not only the estimated value, but also its confidence interval is contained in $(\Pi/2, 3\Pi/2)$ and that the sum of the empirical estimate and its standard deviation are close to 
the true probability.

\begin{table}[ht]
\centering
\vspace{-4mm}
\caption{Number of samples to satisfy Ineq.~\eqref{eq:tk} 
 for Iceland118}
 \begin{tabular}{c|c|c|c|c} 
 \hspace{-2mm}Bound ${\bar\theta}_{ij}$, failure prob. $\Pi$ & MC & ALOE & Pmvnorm & MD-Var\\
 \hline
  \hspace{-2mm}$|{\theta}_{ij}| \le \pi/8$, $\Pi$ = 1.2e-01 & 6.4e+02 & 3.7e+02 & \underline{3.2e+02} & 4.1e+02\\
  \hspace{-2mm}$|{\theta}_{ij}| \le \pi/7$, $\Pi$ = 3.0e-02 & 5.1e+04 & 4.1e+02 & 1.1e+03 & \underline{3.5e+02}\\
  \hspace{-2mm}$|{\theta}_{ij}| \le \pi/6$, $\Pi$ = 2.5e-03  & 6.2e+06 & 4.5e+02 & 6.3e+03 & \underline{3.9e+02} \\
  \hspace{-2mm}$|{\theta}_{ij}| \le \pi/5$, $\Pi$ = 2.6e-05  & 8.9e+10 & 3.3e+02 & 1.4e+04 & \underline{2.1e+02}\\
 \end{tabular}
 \vspace{-4mm}\label{tab:sample-compX}
\end{table}

Table~\ref{tab:emp} shows failure probability estimates and their standard deviations for the algorithms based on $N = 200$ samples on various PandaPower~\cite{pandapower.2018} power grids. In all the presented cases except for the Iceland grid, we set the standard deviations of output powers of generators to $0.25$ of their average values. For the Iceland test case, we use $0.1$ instead of $0.25.$ 
Pmvnorm~\cite{genz2020package} did not terminate on Polish 3120sp case after an hour of computations which we indicated as N/A. All other methods terminate in less than a minute. The proposed algorithms reduce variance and are more computationally efficient than the state-of-the-art ALOE and pmvnorm. Fig.~\ref{fig:weights118}~shows~a~substantial change in hyperplane weight assignment made by MD-Var.

\begin{figure}[t]
    \centering
    \vspace{-1mm}
    \includegraphics[width=.45\textwidth]{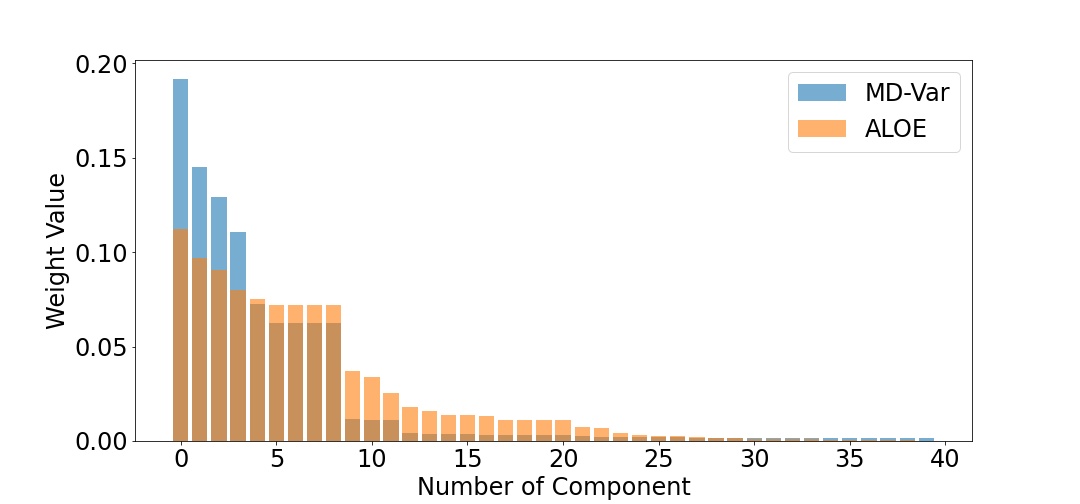}
    \vspace{-2mm}
    \caption{Weights of mixture distribution \eqref{eq:mix} assigned by ALOE and MD-Var. Iceland 118 case with maximum phase angle difference $\pi/3$ and a standard deviation of power injection on generators equal to $0.1$ of the average value.}
    \label{fig:weights118}
    \vspace{-4mm}
\end{figure}

\begin{table}[t]
    \centering
    \caption{Failure probability estimation for power grids.}
    \label{tab:emp}
    {\scriptsize{
    \begin{tabular}{l|l|l|l|l|l|l}
          Estimate, $\hat\Pi$ & $\bar\theta$ & $\Pi$ &  \text{ALOE} & MD-Var & MD-KL & pmvnorm\\
          \hline
          \multicolumn{7}{c}{IEEE 30}\\
          \hline
$\hat\Pi \times$ 1e+15 & $\pi/4$\hspace{-4mm}&8.2\!\! & 8.2 $\pm$ 0.9 & 8.2 $\pm$ 0.0 & 8.2 $\pm$  0.0 & 8.2 $\pm$ 1.2 \\
$\hat\Pi \times$ 1e+06 & $\pi/6$\hspace{-4mm}&5.8\!\! & 5.8 $\pm$ 0.7 & 5.8 $\pm$ 0.0 & 5.8 $\pm$ 0.0 & 5.8 $\pm$ 1.0 \\
$\hat\Pi \times$ 1e+04 & $\pi/7$\hspace{-4mm}&2.9\!\! & 2.9 $\pm$ 0.3 & 2.9 $\pm$ 0.0 & 2.9 $\pm$ 0.0 & 2.9 $\pm$ 0.5\\
$\hat\Pi \times$ 1e+03 & $\pi/8$\hspace{-4mm}&3.1\!\! & 3.1 $\pm$ 0.4 & 3.1 $\pm$ 0.0 & 3.1 $\pm$ 0.1 & 3.1 $\pm$ 0.4\\
\hline
\multicolumn{7}{c}{IEEE 57}\\
\hline
$\hat\Pi \times$ 1e+03 & $\pi/2$\hspace{-4mm}&8.8\!\! & 9.1 $\pm$ 0.8 & 8.7 $\pm$ 0.0 & 8.8 $\pm$ 0.0 & 8.9 $\pm$ 1.2\\
$\hat\Pi \times$ 1e+02 & $\pi/3$\hspace{-4mm}&8.4\!\! & 8.5 $\pm$ 1.1 & 8.4 $\pm$ 0.1 & 8.3 $\pm$ 0.5 & 9.0 $\pm$ 0.9\\
          \hline
          \multicolumn{7}{c}{Iceland 118}\\
          \hline
$\hat\Pi \times$ 1e+09 & $\pi/2$\hspace{-4mm}&6.2\!\! & 6.2 $\pm$ 0.1 & 6.1 $\pm$ 0.0 & 6.1 $\pm$ 0.0 & 5.7 $\pm$ 2.6\\
$\hat\Pi \times$ 1e+04 & $\pi/3$\hspace{-4mm}&2.8\!\! & 3.0 $\pm$0.0 & 2.9 $\pm$ 0.0 & 2.9 $\pm$ 0.0 & 2.8 $\pm$ 1.4\\
$\hat\Pi \times$ 1e+02 & $\pi/4$\hspace{-4mm}&1.1\!\! & 1.1 $\pm$ 0.2 & 1.1 $\pm$ 0.0 & 1.1 $\pm$ 0.0 & 1.1 $\pm$ 0.2\\
$\hat\Pi \times$ 1e+01 & $\pi/6$\hspace{-4mm}&1.4\!\! & 1.4 $\pm$ 0.2 & 1.4 $\pm$ 0.0 & 1.3 $\pm$ 0.0 & 1.4 $\pm$ 0.1\\
\hline 
          \multicolumn{7}{c}{Illinois 200}\\
\hline
$\hat\Pi\times$ 1e+12 & $\pi/4$\hspace{-4mm}&7.9\!\! & 7.9 $\pm$ 0.9 & 7.9 $\pm$ 0.0 & 7.9 $\pm$ 0.0 & 7.9 $\pm$ 3.3\\
$\hat\Pi\times$ 1e+04 & $\pi/6$\hspace{-4mm}&1.1\!\! & 1.1 $\pm$ 0.1 & 1.1 $\pm$ 0.0 & 1.1 $\pm$ 0.0 & 1.1 $\pm$ 0.3\\
$\hat\Pi\times$ 1e+03 & $\pi/7$\hspace{-4mm}&2.3\!\! & 2.3 $\pm$ 0.2 & 2.3 $\pm$ 0.0 & 2.3 $\pm$ 0.0 & 2.3 $\pm$ 0.3\\
$\hat\Pi\times$ 1e+02 & $\pi/8$\hspace{-4mm}&1.5\!\! & 1.5 $\pm$ 0.1 & 1.5 $\pm$ 0.0 & 1.5 $\pm$ 0.0 & 1.5 $\pm$ 0.1\\
\hline
\multicolumn{7}{c}{Polish 3120sp}\\
\hline
$\hat\Pi\times$ 1e+13 & $\pi/2$\hspace{-4mm}&3.7\!\! & 3.7 $\pm$ 0.4 & 3.7 $\pm$ 0.0 & 3.7 $\pm$ 0.0 & N/A\\
$\hat\Pi\times$ 1e+04 & $\pi/3$\hspace{-4mm}&1.2\!\! & 1.2 $\pm$ 0.1 & 1.2 $\pm$ 0.0 & 1.2 $\pm$ 0.0 & N/A\\
$\hat\Pi\times$ 1e+02 & $\pi/4$\hspace{-4mm}&3.4\!\! & 3.4 $\pm$ 0.5 & 3.4 $\pm$ 0.3 & 3.4 $\pm$ 0.6 & N/A
    \end{tabular}}
    }
    \vspace{-4mm}
\end{table}

\section{Conclusion}\label{sec:conclusion}

Importance sampling can be a useful tool for real-time reliability assessment in direct current power grids. We proposed an algorithm that, first, constructs a physics-informed mixture distribution for importance sampling, and, second, utilizes convex optimization to adjust the weights of the mixture.  The method outperforms state-of-the-art algorithms in accuracy and efficiency of reliability assessment. We hope that this approach can be further used for optimization and control in power grids. 

\bibliographystyle{IEEEtran}
\bibliography{biblio}

\end{document}